\documentclass{article}
\usepackage{amsfonts}
\usepackage{amsmath}

\newtheorem{theorem}{Theorem}

\newtheorem{corollary}[theorem]{Corollary}

\newtheorem{proposition}[theorem]{Proposition}
\newtheorem{remark}[theorem]{Remark}

\newenvironment{proof}[1][Proof]{\noindent\textbf{#1.} }{\ \rule{0.5em}{0.5em}}
\input{tcilatex}
\begin{document}

\title{Investigating Exponential and Geometric Polynomials with Euler-Seidel
Algorithm}
\author{Ayhan Dil and Veli Kurt \\
%EndAName
Department of Mathematics, Akdeniz University, 07058 Antalya Turkey\\
adil@akdeniz.edu.tr, vkurt@akdeniz.edu.tr}
\maketitle

\begin{abstract}
In this paper we use Euler-Seidel matrices method to find out some
properties of exponential and geometric polynomials and numbers. Some known
results are reproved and some new results are obtained.


\end{abstract}

\section{Introduction.}

This work is based on Euler-Seidel matrices method $\left( \cite{Dumont}%
\right) $ which is related to algorithms, combinatorics and generating
functions. This method is quite useful to investigate properties of some
special numbers and polynomials.

In this work we use this method to find out some interesting results of
exponential (or Bell) and geometric (or Fubini) polynomials and numbers.
Although some results in this paper are known, this method provides
different proofs as well as new identities.

We first consider a given sequence $(a_{n})_{n\geq 0}$, and then determine
the Euler-Seidel matrix corresponding to this sequence is recursively by the
formulae

\begin{eqnarray}
a_{n}^{0} &=&a_{n}\text{,}\quad (n\geq 0),  \label{1} \\
a_{n}^{k} &=&a_{n}^{k-1}+a_{n+1}^{k-1}\text{,}\quad (n\geq 0,\,k\geq 1) 
\notag
\end{eqnarray}%
where $a_{n}^{k}$ represents the $k$th row and $n$th column entry. From
relation $(\ref{1})$, it can be seen that the first row and the first column
can be transformed into each other via the well known binomial inverse pair
as:%
\begin{equation}
a_{0}^{n}=\sum_{k=0}^{n}\binom{n}{k}a_{k}^{0}\text{,}  \label{2}
\end{equation}%
and%
\begin{equation}
a_{n}^{0}=\sum_{k=0}^{n}\binom{n}{k}(-1)^{n-k}a_{0}^{k}.  \label{2+}
\end{equation}

Euler $\left( \cite{Euler}\right) $ deduced the following proposition which
states a connection between the ordinary generating functions of the initial
sequence $(a_{n})_{n\geq 0}=(a_{n}^{0})_{n\geq 0}$ and the first column $%
(a_{0}^{n})_{n\geq 0}$.

\begin{proposition}[Euler]
Let 
\begin{equation*}
a(t)=\sum_{n=0}^{\infty }a_{n}^{0}t^{n}
\end{equation*}%
be the generating function of the initial sequence $(a_{n}^{0})_{n\geq 0}$.
Then the generating function of the sequence $(a_{0}^{n})_{n\geq 0}$ is 
\begin{equation}
\overline{a}(t)=\sum_{n=0}^{\infty }a_{0}^{n}t^{n}=\frac{1}{1-t}a\left( 
\frac{t}{1-t}\right) .  \label{3}
\end{equation}
\end{proposition}

A similar statement was proved by Seidel $\left( \cite{Seidel}\right) $ with
respect to the exponential generating function.

\begin{proposition}[Seidel]
\label{P}Let 
\begin{equation*}
A(t)=\sum_{n=0}^{\infty }a_{n}^{0}\frac{t^{n}}{n!}
\end{equation*}%
be the exponential generating function of the initial sequence $%
(a_{n}^{0})_{n\geq 0}$. Then the exponential generating function of the
sequence $(a_{0}^{n})_{n\geq 0}$ is 
\begin{equation}
\overline{A}(t)=\sum_{n=0}^{\infty }a_{0}^{n}\frac{t^{n}}{n!}=e^{t}A(t).
\label{4}
\end{equation}
\end{proposition}

Dumont $\left( \cite{Dumont}\right) $, presented several examples of Euler-\
Seidel matrices mainly\ using Bernoulli, Euler, Genocchi, exponential (Bell)
and tangent numbers. He also attempted to give the polynomial extension of
Euler- Seidel matrices method. In \cite{Diletal}, Dil et al obtained some
identities on Bernoulli and allied polynomials by introducing polynomial
extension of this matrices method. By choosing the initial sequence from the
elements of $\mathbb{Z}_{p}$ ($p$ is prime), Dil and Kurt $\left( \cite{DK}%
\right) $ were interested in the type of Euler-Seidel matrices on $\mathbb{Z}%
_{p}$. \cite{MD} contains detailed study on harmonic and hyperharmonic
numbers using Euler-\ Seidel matrices method. More than that the reader also
can find in \cite{MD}, some results on $r-$Stirling numbers and a new
characterization of the Fibonacci sequence. In \cite{DM},\ Dil and Mez\H{o}%
,\ presented another algorithm which depends on a recurrence relation and
two initial sequences. Using the algorithm which is symmetric respect to the
rows and columns, they obtained some relations between Lucas sequences and
incomplete Lucas sequences. Besides Fibonacci and Lucas numbers they also
investigated hyperharmonic numbers.

In this paper we consider Euler- Seidel matrices method for some
combinatorial numbers and polynomials. This method is relatively easier than
most of combinatorial methods to investigate the structure of such numbers
and polynomials.

\section{Definitions and Notation.}

Now we give a summary about some special numbers and polynomials which we
will need later.

\textbf{Stirling numbers of the second kind.}

Stirling numbers of the second kind\ $%
\begin{Bmatrix}
n \\ 
k%
\end{Bmatrix}%
$ are defined by means of generating functions as follows $\left( \cite{AS,C}%
\right) $:%
\begin{equation}
\sum_{n=0}^{\infty }%
\begin{Bmatrix}
n \\ 
k%
\end{Bmatrix}%
\frac{x^{n}}{n!}=\frac{\left( e^{x}-1\right) ^{k}}{k!}.  \label{5}
\end{equation}

\textbf{Exponential polynomials.}

Exponential polynomials (or single variable Bell polynomials) $\phi
_{n}\left( x\right) $ are defined by $(\cite{BL1, Ri, R})$ as%
\begin{equation}
\phi _{n}\left( x\right) :=\sum_{k=0}^{n}\QATOPD\{ \} {n}{k}x^{k}.  \label{8}
\end{equation}

We refer $\cite{B2}$\ for comprehensive information on exponential
polynomials.

The first few exponential polynomials are:%
\begin{equation}
\begin{tabular}{|l|}
\hline
$\phi _{0}\left( x\right) =1\text{,}$ \\ \hline
$\phi _{1}\left( x\right) =x\text{,}$ \\ \hline
$\phi _{2}\left( x\right) =x+x^{2}\text{,}$ \\ \hline
$\phi _{3}\left( x\right) =x+3x^{2}+x^{3}\text{,}$ \\ \hline
$\phi _{4}\left( x\right) =x+7x^{2}+6x^{3}+x^{4}\text{.}$ \\ \hline
\end{tabular}
\label{L1}
\end{equation}

Exponential generating function of exponential polynomials is given by $%
\left( \cite{C}\right) $,%
\begin{equation}
\sum_{n=0}^{\infty }\phi _{n}\left( x\right) \frac{t^{n}}{n!}=e^{x\left(
e^{t}-1\right) }.  \label{7}
\end{equation}

The well known exponential numbers (or Bell numbers) $\phi _{n}$ $\left( 
\cite{BL2, C, CG, T}\right) $\ are obtained by setting $x=1$ in $\left( \ref%
{8}\right) $,$\,$i.e%
\begin{equation}
\phi _{n}:=\phi _{n}\left( 1\right) =\sum_{k=0}^{n}\QATOPD\{ \} {n}{k}.
\label{6}
\end{equation}%
The first few exponential numbers are:%
\begin{equation}
\phi _{0}=1\text{, }\phi _{1}=1\text{, }\phi _{2}=2\text{, \ }\phi _{3}=5%
\text{, \ }\phi _{4}=15\text{.}  \label{L2}
\end{equation}

Readers might also consult the lengthy bibliography of Gould $\left( \cite%
{GU}\right) $, where several papers and books are listed about exponential
numbers.

Following recurrence relations that we reprove with Euler- Seidel matrices
method hold for exponential polynomials $\left( \cite{R}\right) $,%
\begin{equation}
\phi _{n+1}\left( x\right) =x\left( \phi _{n}\left( x\right) +\phi
_{n}^{^{\prime }}\left( x\right) \right)  \label{10}
\end{equation}%
and%
\begin{equation}
\phi _{n+1}\left( x\right) =x\sum_{k=0}^{n}\binom{n}{k}\phi _{k}\left(
x\right) .  \label{11}
\end{equation}

\textbf{Geometric polynomials and numbers.}

Geometric polynomials (also known as Fubini polynomials) are defined as
follows $\left( \cite{B}\right) $:%
\begin{equation}
F_{n}\left( x\right) =\sum_{k=0}^{n}%
\begin{Bmatrix}
n \\ 
k%
\end{Bmatrix}%
k!x^{k}.  \label{13}
\end{equation}%
By setting $x=1$ in $\left( \ref{13}\right) $ we obtain geometric (or
preferential arrangement-, or Fubini-) numbers $F_{n}$, $\left( \cite{Da, Gr}%
\right) $ as%
\begin{equation}
F_{n}:=F_{n}\left( 1\right) =\sum_{k=0}^{n}%
\begin{Bmatrix}
n \\ 
k%
\end{Bmatrix}%
k!.  \label{14}
\end{equation}

The exponential generating functions of geometric polynomials is given by $%
\left( \cite{B}\right) $%
\begin{equation}
\frac{1}{1-x\left( e^{t}-1\right) }=\sum_{n=0}^{\infty }F_{n}\left( x\right) 
\frac{t^{n}}{n!}.  \label{15}
\end{equation}

Let us give a short list of these polynomials and numbers as follows%
\begin{equation*}
\begin{tabular}{|l|}
\hline
$F_{0}\left( x\right) =1\text{,}$ \\ \hline
$F_{1}\left( x\right) =x\text{,}$ \\ \hline
$F_{2}\left( x\right) =x+2x^{2}\text{,}$ \\ \hline
$F_{3}\left( x\right) =x+6x^{2}+6x^{3}\text{,}$ \\ \hline
$F_{4}\left( x\right) =x+14x^{2}+36x^{3}+24x^{4}\text{.}$ \\ \hline
\end{tabular}%
\end{equation*}%
and

\begin{equation*}
F_{0}=1\text{, \ }F_{1}=1\text{, \ }F_{2}=3\text{, \ }F_{3}=13\text{, \ }%
F_{4}=75\text{.}
\end{equation*}

Geometric and exponential polynomials are connected in \cite{B} by the
relation%
\begin{equation}
F_{n}\left( x\right) =\int_{0}^{\infty }\phi _{n}\left( x\lambda \right)
e^{-\lambda }d\lambda .  \label{16}
\end{equation}

Now we will state our results.

\section{Results obtained by matrix method.}

Although we define Euler-Seidel matrices as matrices of numbers, we can
consider entries of these matrices as polynomials also $\left( \text{see }%
\cite{Diletal}\right) $. Thus, the generating functions that we mention in
the statement of Seidel's proposition, turn out to be two variables
generating functions. Therefore from now on when we consider these
generating functions as exponential generating functions of numbers we use
the notations $A\left( t\right) $ and $\overline{A}\left( t\right) $,
otherwise for the polynomial case we use the notations $A\left( t,x\right) $
and $\overline{A}\left( t,x\right) $. Using these notations relation $\left( %
\ref{4}\right) $ becomes%
\begin{equation}
\overline{A}\left( t,x\right) =e^{t}A\left( t,x\right) .  \label{16'}
\end{equation}

\subsection{Results on Stirling numbers of the second kind.}

Setting the initial sequence of an Euler-Seidel matrix as the sequence of
Stirling numbers of the second kind, i.e. $a_{n}^{0}=%
\begin{Bmatrix}
n \\ 
m%
\end{Bmatrix}%
$ where $m$ is a fixed nonnegative integer, we get the exponential
generating function of the first row as%
\begin{equation*}
A\left( t\right) =\frac{\left( e^{t}-1\right) ^{m}}{m!}.
\end{equation*}%
Thus we obtain from $\left( \ref{2}\right) $ the result%
\begin{equation}
a_{0}^{n}=\sum_{k=0}^{n}\binom{n}{k}%
\begin{Bmatrix}
k \\ 
m%
\end{Bmatrix}%
.  \label{17'}
\end{equation}%
Now we proceed as follows; by using Proposition $\left( \ref{P}\right) $ we
obtain another form of "$a_{0}^{n}$" and then we combine it with $\left( \ref%
{17'}\right) $.

Equation $\left( \ref{4}\right) $ yields%
\begin{equation*}
\overline{A}\left( t\right) =\sum_{n=0}^{\infty }a_{0}^{n}\frac{t^{n}}{n!}%
=e^{t}\frac{\left( e^{t}-1\right) ^{m}}{m!}
\end{equation*}%
which can equally well be written%
\begin{equation}
\overline{A}\left( t\right) =\frac{d}{dt}\frac{\left( e^{t}-1\right) ^{m+1}}{%
\left( m+1\right) !}.  \label{17''}
\end{equation}%
Comparison of coefficients of $t^{n}$ in equation $\left( \ref{17''}\right) $
yields the following result:%
\begin{equation}
a_{0}^{n}=%
\begin{Bmatrix}
n+1 \\ 
m+1%
\end{Bmatrix}%
.  \label{18}
\end{equation}%
Thus we obtain the following proposition.

\begin{proposition}
Let $m$ be a nonnegative integer then we have,
\end{proposition}

\begin{equation}
\sum_{k=0}^{n}\binom{n}{k}%
\begin{Bmatrix}
k \\ 
m%
\end{Bmatrix}%
=%
\begin{Bmatrix}
n+1 \\ 
m+1%
\end{Bmatrix}
\label{19}
\end{equation}%
and%
\begin{equation}
\sum_{k=0}^{n}\binom{n}{k}\left( -1\right) ^{n-k}%
\begin{Bmatrix}
k+1 \\ 
m+1%
\end{Bmatrix}%
=%
\begin{Bmatrix}
n \\ 
m%
\end{Bmatrix}%
.  \label{20}
\end{equation}

\begin{proof}
The equalities $\left( \ref{17'}\right) $ and $\left( \ref{18}\right) $
constitute the first result stated in proposition. Then equation $\left( \ref%
{20}\right) $ directly follows\ from $\left( \ref{2+}\right) $ and $\left( %
\ref{19}\right) $.
\end{proof}

Relations $\left( \ref{19}\right) $\ and $\left( \ref{20}\right) $\ can be
found in \cite{GKP}, respectively as the equations (6.15) and (6.17) on page
265.

\subsection{Results on exponential polynomials\ and numbers.}

\subsubsection{Exponential numbers.}

Let us construct an Euler-Seidel matrix with the initial sequence $\left(
a_{n}^{0}\right) _{n\geq 0}=\left( \phi _{n}\right) _{n\geq 0}$. Then we get
the following Euler-Seidel matrix

\begin{equation*}
\left[ 
\begin{array}{ccccccc}
1 & 1 & 2 & 5 & 15 & 52 & \cdots \\ 
2 & 3 & 7 & 20 & \cdots &  &  \\ 
5 & 10 & 27 & \cdots &  &  &  \\ 
15 & 37 & \cdots &  &  &  &  \\ 
52 & \cdots &  &  &  &  &  \\ 
\cdots &  &  &  &  &  & 
\end{array}%
\right]
\end{equation*}%
From this matrix we observe that $a_{0}^{n}=$ $\phi _{n+1}$. Now we prove
this observation using generating functions.

$\left( a_{n}^{0}\right) _{n\geq 0}=\left( \phi _{n}\right) _{n\geq 0}$,
from which it follows that%
\begin{equation*}
A\left( t\right) =e^{e^{t}-1}.
\end{equation*}%
Equation $\left( \ref{4}\right) $ permits us to write%
\begin{equation}
\overline{A}\left( t\right) =e^{e^{t}+t-1}=\frac{d}{dt}\left(
e^{e^{t}-1}\right) =\sum_{n=0}^{\infty }\phi _{n+1}\frac{t^{n}}{n!}.
\label{20'}
\end{equation}%
Comparision of the coefficients of both sides in $\left( \ref{20'}\right) $
gives the desired result%
\begin{equation}
a_{0}^{n}=\phi _{n+1}.  \label{20+}
\end{equation}

This leads to the following proposition.

\begin{proposition}
\label{B}We have%
\begin{equation}
\phi _{n+1}=\sum_{k=0}^{n}\binom{n}{k}\phi _{k}  \label{21}
\end{equation}%
and%
\begin{equation}
\phi _{n}=\sum_{k=0}^{n}\binom{n}{k}\left( -1\right) ^{n-k}\phi _{k+1}.
\label{22}
\end{equation}
\end{proposition}

\begin{proof}
Equation $\left( \ref{21}\right) $\ follows from $\left( \ref{2}\right) $
and $\left( \ref{20+}\right) $. Hence considering $\left( \ref{2+}\right) $
and $\left( \ref{21}\right) $ together we obtain $\left( \ref{22}\right) $.
\end{proof}

The identity $\left( \ref{21}\right) $ can be found in \cite{GKP} on page
373 and $\left( \ref{22}\right) $ is inverse binomial transform of identity $%
\left( \ref{21}\right) $.

\subsubsection{Exponential polynomials.}

Setting the initial sequence of an Euler-Seidel matrix as the sequence of
exponential polynomials, i.e. $\left( a_{n}^{0}\right) _{n\geq 0}=\left(
\phi _{n}\left( x\right) \right) _{n\geq 0}$ we get following Euler- Seidel
matrix,

\begin{equation*}
\left[ 
\begin{array}{ccccc}
1 & x & x+x^{2} & x+3x^{2}+x^{3} & \cdots \\ 
1+x & 2x+x^{2} & 2x+4x^{2}+x^{3} & \cdots &  \\ 
1+3x+x^{2} & 4x+5x^{2}+x^{3} & \cdots &  &  \\ 
1+7x+6x^{2}+x^{3} & \cdots &  &  &  \\ 
\cdots &  &  &  & 
\end{array}%
\right]
\end{equation*}%
It seems that, $xa_{0}^{n}=\phi _{n+1}\left( x\right) $. Now we prove this
fact.

With the aid of Proposition $\left( \ref{P}\right) $, we can write%
\begin{equation}
\overline{A}\left( t,x\right) =e^{t}e^{x\left( e^{t}-1\right) }=\frac{1}{x}%
\frac{d}{dt}e^{x\left( e^{t}-1\right) }.  \label{22'}
\end{equation}%
Comparision of the coefficients of the both sides in $\left( \ref{22'}%
\right) $ gives the desired result:%
\begin{equation}
xa_{0}^{n}=\phi _{n+1}\left( x\right)  \label{22+}
\end{equation}%
from which the next proposition follows.

\begin{proposition}
We have
\end{proposition}

\begin{equation}
\phi _{n+1}\left( x\right) =x\sum_{k=0}^{n}\binom{n}{k}\phi _{k}\left(
x\right)  \label{23'}
\end{equation}
and%
\begin{equation}
x\phi _{n}\left( x\right) =\sum_{k=0}^{n}\binom{n}{k}\left( -1\right)
^{n-k}\phi _{k+1}\left( x\right) .  \label{24}
\end{equation}

\begin{proof}
Proof is like that of Proposition $\left( \ref{B}\right) $.
\end{proof}

Here we give a new proof of the equation $\left( \ref{11}\right) $ by using
Euler- Seidel matrices method.

It is clear that equations $\left( \ref{23'}\right) $ and $\left( \ref{24}%
\right) $ are the generalizations of equations $\left( \ref{21}\right) $ and 
$\left( \ref{22}\right) ,$\ respectively.

Now with the help of generating functions technique we derive some relations
for exponential polynomials.

Firstly we give a new proof of the equation $\left( \ref{10}\right) .$

\begin{proposition}
Let $\phi _{n}^{^{\prime }}\left( x\right) $ denote the derivative of the $n$%
th exponential polynomial $\phi _{n}\left( x\right) ,$ with respect to the
variable $x$. Then the equation%
\begin{equation}
\phi _{n+1}\left( x\right) =x\left( \phi _{n}\left( x\right) +\phi
_{n}^{^{\prime }}\left( x\right) \right)  \label{25'}
\end{equation}%
holds.
\end{proposition}

\begin{proof}
Deriving both sides of the equation $\left( \ref{7}\right) $ respect to the $%
x$ we get%
\begin{equation*}
\sum_{n=0}^{\infty }\phi _{n}^{^{\prime }}\left( x\right) \frac{t^{n}}{n!}%
=e^{t}e^{x\left( e^{t}-1\right) }-e^{x\left( e^{t}-1\right) }
\end{equation*}%
which combines with $\left( \ref{22'}\right) $ to give%
\begin{equation*}
\sum_{n=0}^{\infty }\phi _{n}^{^{\prime }}\left( x\right) \frac{t^{n}}{n!}=%
\overline{A}\left( t,x\right) -A\left( t,x\right) .
\end{equation*}%
Then the last equation gives the desired result by comparing coefficients.
\end{proof}

\begin{corollary}
Exponential polynomials and their derivatives satisfy following symmetric
equation%
\begin{equation}
\sum_{k=0}^{n-1}\binom{n}{k}\left( -1\right) ^{k}\phi _{k}\left( x\right)
=\sum_{k=1}^{n}\binom{n}{k}\left( -1\right) ^{k-1}\phi _{k}^{^{\prime
}}\left( x\right) .  \label{26}
\end{equation}
\end{corollary}

\begin{proof}
Employing $\left( \ref{25'}\right) $ in the equation $\left( \ref{24}\right) 
$, we obtain $\left( \ref{26}\right) $.
\end{proof}

\subsection{Results on geometric polynomials\ and numbers}

This part of our work contains some relations on geometric numbers and
polynomials, most of which seems to be new.

\subsubsection{Geometric numbers}

Let us set the initial sequence of an Euler-Seidel matrix as the sequence of
geometric numbers, i.e. $\left( a_{n}^{0}\right) _{n\geq 0}=\left(
F_{n}\right) _{n\geq 0}$. Then we have%
\begin{equation*}
\left[ 
\begin{array}{cccccc}
1 & 1 & 3 & 13 & 75 & \cdots \\ 
2 & 3 & 7 & 20 & \cdots &  \\ 
6 & 10 & 27 & \cdots &  &  \\ 
26 & 37 & \cdots &  &  &  \\ 
150 & \cdots &  &  &  &  \\ 
\cdots &  &  &  &  & 
\end{array}%
\right]
\end{equation*}%
Considering the first row and the first column, we observe that $%
a_{0}^{n}=2F_{n}$, $n\geq 1$. We need a proof of this fact.

Proposition $\left( \ref{P}\right) $ permits us to write%
\begin{equation*}
\overline{A}\left( t\right) =\sum_{n=0}^{\infty }a_{0}^{n}\frac{t^{n}}{n!}=%
\frac{e^{t}}{2-e^{t}}=2\frac{1}{2-e^{t}}-1=\sum_{n=1}^{\infty }2F_{n}\frac{%
t^{n}}{n!}+1\text{.}
\end{equation*}%
Now comparision of the coefficients of the both sides gives $%
a_{0}^{n}=2F_{n} $ where $n\geq 1$.

Then we may summarize the results so far obtained in:

\begin{proposition}
\begin{equation}
2F_{n}=\sum_{k=0}^{n}\binom{n}{k}F_{k}\text{ or equally }F_{n}=%
\sum_{k=0}^{n-1}\binom{n}{k}F_{k}  \label{27}
\end{equation}%
and%
\begin{equation}
F_{n}=2\sum_{k=0}^{n}\binom{n}{k}\left( -1\right) ^{n-k}F_{k}.  \label{28}
\end{equation}
\end{proposition}

\subsubsection{Geometric polynomials}

Let us set the initial sequence of an Euler-Seidel matrix as the sequence of
geometric polynomials, i.e. $\left( a_{n}^{0}\right) _{n\geq 0}=\left(
F_{n}\left( x\right) \right) _{n\geq 0}$. Thus we obtain from $\left( \ref{4}%
\right) $ the result%
\begin{equation*}
A\left( t,x\right) =\sum_{n=0}^{\infty }F_{n}\left( x\right) \frac{t^{n}}{n!}%
=\frac{1}{1-x\left( e^{t}-1\right) }
\end{equation*}%
and%
\begin{equation}
\overline{A}\left( t,x\right) =\frac{e^{t}}{1-x\left( e^{t}-1\right) }.
\label{28'}
\end{equation}%
Then, derivative respect to the $t$ yields%
\begin{equation*}
\overline{A}\left( t,x\right) =\left[ \frac{1}{x}-\left( e^{t}-1\right) %
\right] \frac{d}{dt}A\left( t,x\right)
\end{equation*}%
which can equally well be written%
\begin{equation*}
\overline{A}\left( t,x\right) =\sum_{n=0}^{\infty }\left[ \frac{%
F_{n+1}\left( x\right) }{x}+F_{n+1}\left( x\right) -\sum_{k=0}^{n}\binom{n}{k%
}F_{k+1}\left( x\right) \right] \frac{t^{n}}{n!}.
\end{equation*}%
The next equation follows by equating coefficients of $\frac{t^{n}}{n!}$\ in
the preceding equation,%
\begin{equation}
a_{0}^{n}=\frac{F_{n+1}\left( x\right) }{x}-\sum_{k=1}^{n}\binom{n}{k-1}%
F_{k}\left( x\right) .  \label{29}
\end{equation}%
We can now prove:

\begin{proposition}
\label{F1}$F_{n}\left( x\right) $ geometric polynomials satisfy the
following recurrence relation%
\begin{equation}
F_{n}\left( x\right) =x\sum_{k=0}^{n-1}\binom{n}{k}F_{k}\left( x\right) .
\label{30}
\end{equation}
\end{proposition}

\begin{proof}
In view of $\left( \ref{2}\right) $,\ equation $\left( \ref{29}\right) $\
shows validity of $\left( \ref{30}\right) $.
\end{proof}

\begin{remark}
As a special case, we get $\left( \ref{27}\right) $ by setting $x=1$ in $%
\left( \ref{30}\right) $.
\end{remark}

\begin{corollary}
\begin{equation}
F_{n+1}\left( x\right) =\frac{x}{1+x}\sum_{k=0}^{n}\binom{n}{k}\left[
F_{k}\left( x\right) +F_{k+1}\left( x\right) \right] .  \label{31}
\end{equation}
\end{corollary}

Now we give two more recurrence relations for geometric polynomials.

\begin{proposition}
\label{F2}Let $F_{k}^{\prime }\left( x\right) $ denotes the first derivative
of $F_{k}\left( x\right) ,$ respect to the variable $x$. Then we have%
\begin{equation}
F_{n+1}\left( x\right) =x\sum_{k=0}^{n}\binom{n}{k}\left[ F_{k}\left(
x\right) +xF_{k}^{\prime }\left( x\right) \right] .  \label{32}
\end{equation}
\end{proposition}

\begin{proof}
We may use $\left( \ref{16}\right) $ and $\left( \ref{23'}\right) $ to
conclude that,%
\begin{equation*}
F_{n+1}\left( x\right) =x\sum_{k=0}^{n}\binom{n}{k}\int_{0}^{\infty }\phi
_{k}\left( x\lambda \right) \lambda e^{-\lambda }d\lambda
\end{equation*}%
from which it follows that%
\begin{equation*}
F_{n+1}\left( x\right) =x\sum_{k=0}^{n}\binom{n}{k}\sum_{r=0}^{k}%
\begin{Bmatrix}
k \\ 
r%
\end{Bmatrix}%
\left( r+1\right) !x^{r}.
\end{equation*}%
This can equally well be written by means of derivative as%
\begin{equation*}
F_{n+1}\left( x\right) =x\frac{d}{dx}x\sum_{k=0}^{n}\binom{n}{k}F_{k}\left(
x\right)
\end{equation*}%
which completes the proof.
\end{proof}

We have the following symmetric relations between geometric polynomials and
their derivatives.

\begin{corollary}
\begin{equation}
\sum_{k=0}^{n}\binom{n}{k}xF_{k}^{\prime }\left( x\right) =\sum_{k=1}^{n}%
\binom{n}{k-1}F_{k}\left( x\right)  \label{33}
\end{equation}
\end{corollary}

\begin{proof}
Combining results of Proposition $\left( \ref{F1}\right) $ and Proposition $%
\left( \ref{F2}\right) $ gives $\left( \ref{33}\right) .$
\end{proof}

\end{document}